\documentclass{amsart}
\usepackage{dino}
\usepackage{csquotes}
\usepackage{todonotes}

\newcommand{\Cantor}{2^{\omega}}

\newcommand{\INF}{\mathrm{INF}}

\newcommand{\E}{\mathrel{E}}
\usepackage{centernot}
\title{Learning equivalence relations on Polish spaces}
\author[Rossegger]{Dino Rossegger}
\author[Slaman]{Theodore Slaman}
\author[Steifer]{Tomasz Steifer}

\address[Rossegger]{Institute of Discrete Mathematics and Geometry, Technische Universit\"at Wien}
\email{dino.rossegger@tuwien.ac.at}

\address[Slaman]{Department of Mathematics, University of California, Berkeley}
\email{slaman@math.berkeley.edu}

\address[Steifer]{Universidad Católica de Chile \& Institute of Fundamental Technological Research, Polish Academy of Sciences}
\email{tsteifer@ippt.pan.pl}

\subjclass{03E15,68Q32, 03D80}
\thanks{The work of Rossegger was supported by the European Union's Horizon 2020 Research and Innovation Programme under the Marie Sk\l{}odowska-Curie grant agreement No. 101026834 — ACOSE. Steifer was supported by the National Science Center grant no. 2021/05/X/ST6/00594 and by the Agencia Nacional de Investigación y Desarrollo grant no. 3230203. Part of this research was conducted when the first and the last author were visiting University of California, Berkeley, where the middle author resides. }

\begin{document}
\maketitle
\begin{abstract}
  We investigate natural variations of behaviourally correct learning and
  explanatory learning---two learning paradigms studied in algorithmic
  learning theory---that allow us to ``learn'' equivalence
  relations on Polish spaces. We give a characterization of the learnable
  equivalence relations in terms of their Borel complexity and show that the behaviourally
  correct and explanatory learnable equivalence relations coincide both in
  uniform and non-uniform versions of learnability and provide a
  characterization of the learnable equivalence relations in terms of their
  Borel complexity. We also show that the set of uniformly learnable equivalence relations
  is $\bPi^1_1$-complete in the codes and study the learnability of several
  equivalence relations arising naturally in logic as a case study.
\end{abstract}

\section{Introduction}

A common scenario in algorithmic learning theory can be viewed in terms of the
following game played in infinitely many stages: 

\begin{displayquote}
Fix a countable sequence of languages
$(l_i)_{i\in\omega}$, $l_i\subseteq \omega$ (possibly with repetitions) and an
index $i_0$. 
At each stage s, an \emph{informant} presents a
\emph{learner} with a word $w_s\in l_{i_0}$ such that $\{w_s: s\in
\omega\}=l_{i_0}$ and the learner has to make an
hypothesis $h_s\in\omega$, guessing which language they are presented with. 
\end{displayquote}
This way of presenting the language $l_{i_0}$ is often referred to as \emph{learning
from text}.
There are several ways to define winning conditions for this learning game. The
first one, giving the learning paradigm often referred to as \emph{explanatory} learning or
\emph{learning in the limit} is due to Gold~\cite{gold1967language}: 
\begin{displayquote}
  The learner wins the game if after finitely many stages they stabilize on a correct
  hypothesis, i.e. if there exists $h=\lim_s h_s$ and $l_h=l_{i_0}$.
\end{displayquote}
Another learning paradigm, \emph{behaviourally correct} (BC) learning, is due to
Feldman~\cite{feldman1972}:
\begin{displayquote}
  The learner wins the game if there is a finite stage $s$ such that at all stages
  $t>s$, $l_{h_t}=l_{i_0}$.
\end{displayquote}
Notice that we do not require that the learner guesses the correct index, they
just need to identify the language up to equality.

Recently, Fokina, Kötzing, and San Mauro~\cite{fokina2019limit} adapted
the explanatory paradigm to learning of the isomorphism relation on countable structures.
As above, we can view this as a game where a sequence of countable
structures $(\A_i)_{i\in\omega}$ is fixed, and the learner, presented with
finite fragments of a structure $\B$, needs to hypothesize to which of the
$\A_i$ $\B$ is isomorphic. The difference to classical
explanatory learning is that the informant may play an isomorphic copy of one
of the structures $\A_i$, or a structure not isomorphic to any of the structures
in the sequence. Bazhenov, Fokina, and San Mauro~\cite{bazhenov2020learning}
characterized learnable sequences of structures as those that have $\Sinf{2}$
quasi Scott sentences in the infinitary logic $L_{\omega_1\omega}$. That is,
sentences $\phi_i$ of the form $\Vvee_{i\in\omega} \exists \bar x
\Wwedge_{j\in\omega} \forall \bar y \psi_{i,j}(\bar x,\bar y)$ where the
$\psi_{i,j}$ are 
quantifier-free such that
$\A_i\models \phi_i$ and if $\A_j\not\cong \A_i$,
then $\A_j\not\models \phi_i$. A result connecting this paradigm with
descriptive set theory is due to Bazhenov, Cipriani and San
Mauro~\cite{bazhenov2023learning}. They showed that a sequence of structures is
learnable if and only if the isomorphism problem on the associated class is
continuously reducible to $E_0$, the tail equivalence relation on infinite
binary strings.

Motivated by the above work, we take a 
different approach to connecting algorithmic learning theory with descriptive set
theory. Our aim is to study the following question:

\begin{displayquote}
  What does it mean for an equivalence relation on a Polish space to be
learnable?
\end{displayquote}

Fokina, Kötzing, and San Mauro's learning paradigm already gives some insight into an equivalence relation on a Polish space. After all, the space of countable structures in a
relational vocabulary can be viewed as a closed subspace of $2^\omega$ and
thus admits a natural Polish topology. However, learnability in their paradigm
is a local property, only giving information on the countable set of isomorphism
classes obtained from the sequence $(\A_i)_{i\in\omega}$. 

Our approach is global, aiming to give information into an equivalence relation on
the whole space. We introduce and study both uniform and non-uniform versions of explanatory learning and behaviourally
correct learning of equivalence relations on Polish spaces. Our main results
show that both the uniform and non-uniform versions of these learning paradigms
are equivalent (\cref{thm:nonuniformbcifflearn}, \cref{thm:bclearning}). The
proof of this result for the non-uniform version is relatively simple, while the
proof of the uniform result is quite involved, using recent results of
Lecomte~\cite{lecomte2020complexity} on Borel equivalence relations and one of
our other main theorems, that an equivalence relation is uniformly (explanatory)
learnable if and only if it is $\bSigma^0_2$. All of the above results are covered
in \cref{sec:uniformlearnability}. 

In \cref{sec:complexityoflearning} we show that the complexity of the set of
learnable $\bPi^0_2$ equivalence relations on $2^\omega$ is $\bPi^1_1$-complete
in the codes. 
In \cref{sec:casestudy} we analyse the learnability of various equivalence
relations that arise naturally in computability and model theory. At last, we
give an example of a natural $\Pi^1_1$ subset of the natural numbers that is
neither $\Sigma^1_1$ nor $\Pi^1_1$-complete.

The proofs of our results use techniques from several areas of mathematical
logic. Our results in \cref{sec:complexityoflearning} rely on
Cohen forcing and results in the theory of Borel equivalence relations, in
\cref{sec:complexityoflearning} we use techniques from higher recursion theory
and effective descriptive set theory, and most proofs in \cref{sec:casestudy}
will seem familiar to computability theorists. We assume that the reader has
some familiarity with these subjects.

All our proofs are written assuming that the underlying Polish space is Cantor
space $2^\omega$---the space of infinite binary sequences under the
product topology. However, the theorems
are stated for general (effective) Polish spaces. Our proofs can be easily
modified to hold in this setting. 
\section{Notions of learnability}\label{sec:uniformlearnability}
\subsection{Non-uniform learnability}
The following definition presents our attempt to obtain a non-uniform generalization of the
explanatory learning framework. 
\begin{definition}\label{def:nonuniformlearning}
  Let $E$ be an equivalence relation on a Polish space $X$ and assume $\omega$ is equipped with the discrete topology. 
  We say that $E$ is \emph{non-uniformly learnable} if for every $\vec x\in
  X^{\omega}$, there are continuous
  functions $l_n(\vec x): X\to \omega$ such that for every $x\in X$, if $xEx_i$ for some $i\in\omega$,
  then $L(x,\vec x)=\lim l_n( \vec x,x)$ exists and $x E x_{L(x,\vec x)}$. We call the partial function $L$ a learner, and write $L(x,\vec x,n)$ for $l_n(\vec x,x)$. 
\end{definition}
If $x\not\E x_i$ for any $x_i\in \vec x$, then the behavior of a potential
learner $L$ is not specified. It might not be defined, or its value might be any
index. We say that a learner $L$ which is defined in this case gives \emph{false
positives}.
\begin{proposition}\label{prop:eqclassesS2sep}
  Let $\E$ be an equivalence relation on a Polish space $X$. Then the following are equivalent.
  \begin{enumerate}
    \item $\E$ is non-uniformly learnable. 
    \item For every $\vec x\in X^\omega$, there are sets $S_i$, $S_i\in\bSigma^0_2(X)$ such
  that for every $i\in\omega$, $[x_i]\subseteq S_i$ and $S_i\cap S_j=\emptyset$
  if $x_i\not \E x_j$.
  \end{enumerate}
  Furthermore, $\E$ is non-uniformly learnable by learners not giving false positives if and only $[x]$ is $\bSigma^0_2$ for every
  $x\in X$.
\end{proposition}
\begin{proof}
  (2)$\implies$(1). Given $\vec x$, let $S_i$ be a countable sequence of
  $\Sigma^0_2(p_i)$ sets covering $[x_i]_{E}$ such that if $x_i\not\E x_j$, then
  $S_i\cap S_j=\emptyset$ and let $R_i$ be recursive
  relations such that
  \[ x\in S_i\iff \exists n\forall m\ R_i(x,p_i,n,m).\]
  Then, for any $x\in X$ and $s\in\omega$, define $L(x,\vec x,s)$ to be the
  minimal $n<s$ such that for all $m<s$ $R_i(x,p_i,n,m)$ if such $n$ exists and
  $L(x,\vec x,s)=s$ otherwise. From the disjointness of the $S_i$ we get that if $j$ is least
  such that $x\E x_j$, then $L(x,\vec x)=j$. Furthermore, for every $s$, $L(-,\vec x,s)$ is
  recursive in $\bigoplus p_i$ and hence continuous. Thus $L$ is a learner
  learning $\vec x$ with respect to $\E$.

  (1)$\implies$(2). Let $L$ be a learner for $\vec x$ and note that the sets
  $L(x,\vec x,n)^{-1}(i)$ are clopen. Then let $C_{n,i}=\{x:
(\forall m>n) L(x,\vec x,m)= i\}$, define
$D_i =\bigcup_n C_{n,i}$ and $S_i=\bigcup_{x_i\E x_j} D_j$. Then the $S_i$ are
$\bSigma^0_2$ and by the definition of non-uniform learnability $[x_i]\subseteq
S_i$. If $L$ learns $\vec x$ without giving false positives, then clearly $S_i=[x_i]$.
\end{proof}

We now turn our attention to a non-uniform adaptation of behavioral correct learning. 
\begin{definition}\label{def:nonuniformbc}
  Let $ \E$ be an equivalence relation on a Polish space $X$ and assume $\omega$ is equipped with the discrete topology. 
  We say that $\E$ is \emph{non-uniformly BC-learnable} if there are continuous
  functions $l_n: X^\omega\times X\to \omega$ such that for $x\in X$ and
  $\vec x=(x_i)_{i\in\omega}\in X^\omega$, if $x\E x_i$ for some $i\in\omega$,
  then for almost every $n$, $x\E x_{l_n(\vec x, x)}$.
\end{definition}
The observant reader might immediately notice that non-uniform BC-learnability
and non-uniform learnability coincide. In order to obtain a learner for $\vec
x$ from a BC-learner, we just have to fix some more information.
\begin{theorem}\label{thm:nonuniformbcifflearn}
  An equivalence relation $\E$ is non-uniformly learnable if and only if it is
  non-uniformly BC-learnable.
\end{theorem}
\begin{proof}
  Clearly, any non-uniformly learnable equivalence relation is non-uniformly
  BC-learnable. On the other hand, say that $(l_n)_{n\in\omega}$ gives a BC
  learner for the sequence $\vec x$ and associate with every $i$, the set
  $e_i=\{ j: x_j\E x_i\}$. Now, define a function $L$ by $L(x,\vec x,n)=\min j [
  j\in e_{l_{n}(x,\vec x)}]$. We claim that $L$ is a learner for $E$ on $\vec
  x$. Clearly, $L$ is continuous on every $n$. Say that $x\E x_i$, then for
  almost every $n$, $l_n(\vec x,x)\in e_i$. So, for almost every $n$, $L(x,\vec
  x,n)=\min j [j \in e_i]$ and thus $L$ learns $\E$ on $\vec x$.
\end{proof}
Recall that an equivalence relation is \emph{countable} if all its equivalence classes
are countable. Countable Borel equivalence relations play an important role in
descriptive set theory and have seen much attention in the past decades,
see~\cite{kechris2019} for a summary of these developments.
Unfortunately, non-uniform learnability does not provide information when it comes to countable
equivalence relations.
\begin{proposition}\label{prop:ctblenonuniform}
    Every countable equivalence relation on a Polish space is non-uniformly
    learnable by learners not giving false positives.
\end{proposition}
\begin{proof}
  Let $\E$ be a countable equivalence relation. Given a hypothesis $\vec{x}$,
  let $\vec{y_i}$ be an enumeration of the equivalence class for $x_i$ and for
  every $n$,
  define the continuous function 
  \[l_n(x)=\begin{cases}(\min i<n) [ (\exists j<n)\ y_{i,j}\restrict n=x\restrict n] &
    \text{if such $i$ exists}\\
  n& \text{otherwise}\end{cases}.\]
  It is not hard to see that the so-defined sequence of functions
  defines a learner.
\end{proof}
\subsection{Uniform learnability}
\begin{definition}\label{def:uniformlearning}
  Let $\E$ be an equivalence relation on a Polish space $X$ and assume $\omega$ is equipped with the discrete topology. 
  We say that $\E$ is \emph{uniformly learnable}, or just \emph{learnable}, if there are continuous
  functions $l_n: X^\omega\times X\to \omega$ such that for $x\in X$ and
  $\vec x=(x_i)_{i\in\omega}\in X^\omega$, if $x\E x_i$ for some $i\in\omega$,
  then $L(x,\vec x)=\lim l_n(\vec x, x)$ exists and $x \E x_{L(x,\vec x)}$. We call the partial function $L$ a \emph{learner}
  for $E$ and write $L(x,\vec x,n)$ for $l_n(\vec x, x)$.
\end{definition}
\begin{definition}\label{def:computablelearner}
    A learner $L=\lim l_n$ is $x$-computable for $x\in 2^\omega$ if 
    there is an $x$-recursive function $f:\omega\to \omega$ such
    that $\Phi_{f(n)}^x=l_n$ where $(\Phi_i)_{i\in \omega}$ is a standard
    enumeration of Turing operators.
\end{definition}
Clearly, every learner is $x$-computable for some $x\in 2^\omega$. The main
result of this section is a characterization of the uniformly learnable
equivalence relations in terms of their Borel complexity. Our proof will
rely on forcing. There are many treatments of forcing in the literature. For a concise overview of forcing both in the computability theory and set theory setting we suggest~\cite{chong2015}. Before we give our proof we note some interactions between
learnable equivalence relations and generics.

\begin{proposition}\label{prop-genericcone}
Let $\E$ be learnable by an $a$-computable learner. If for $x\in 2^\omega$ there exists 
$z\in 2^\omega$, 1-generic relative to $x\oplus a$, such that $x\E z$, then there
exists $p\in 2^{<\omega}$
such that for all $y\in\llbracket p \rrbracket$ we have $x\E y$.
\end{proposition}
\begin{proof}
Suppose that $\E$ is learnable by an $a$-computable learner $L$. Fix a real $x$
and $\vec g$, a sequence of mutually 1-generics relative to $x\bigoplus a$. By
learnability, if for any $g$ in $\vec g$ we have $x\E g$, then $L(x,\vec g)$
converges. Suppose that there exists $k$ and $n_0$ such that for all $m>n_0$ we
have $L(x,\vec g, m)=k$ and fix such $k$ and $n_0$. We will argue that then
there exists a finite $\vec p$ with $p_i\in 2^{<\omega}$ that already forces the convergence. Consider a set $C$ consisting of finite $\vec h$ with $h_i\in 2^{<\omega}$ such that for some $n>n_0$ the learner diverges from $k$, i.e., $L(x,\vec h,n)\neq k$. This set is recursively enumerable in $x$ and $a$ and $\vec g$ has to either meet or avoid it. If no finite $\vec p$ forces the convergence, then every initial segment of $\vec g$ can be extended to an element of $C$. But then $\vec g$ cannot avoid $C$, contradicting the definition of $n_0$. 

So take $\vec p$ such that $\vec p\forces (\forall m>n_0) L(x,\dot{\vec g},m)=k$ and assume without loss of generality that $k<|\vec p|$. Then for any $y\in \llbracket p_k\rrbracket$, $L(x,p_1,\dots, y, \dots, p_{|\vec p|})=k$, since by the use principle for continuous functions if there exists $q\succeq p_k$ such that $L(x,p_1,\dots, q,\dots, p_{|\vec p|},m)=k$ for any $m$, then $L(x,p_1,\dots r,\dots, p_{|\vec p|},m)=k$ for all $r\succeq q$.

\end{proof}
\begin{corollary}
    If $\E$ is learnable by an $a$-computable learner and the $\E$-equivalence
    class of $x$ is countable, then $x$ is not $\E$-equivalent to any $y$ which
    is 1-generic relative to $x\oplus a$.
\end{corollary}

\begin{theorem}\label{thm:sigma2ifflearnable}
Let $\E$ be an equivalence relation. 
Then $\E$
is learnable by an $a$-computable learner if and only if $\E$ is $\Sigma^0_2(a)$.
\end{theorem}
\begin{proof}
    ($\Rightarrow$) Suppose that $\E$ is learnable by an $a$-computable
    learner $L$. Take arbitrary $x$ and $y$ such that $x\E y$ and let $\vec g$
    be a countable sequence of reals mutually sufficiently generic relative to $x\oplus y\oplus  a$.  We analyze the behavior of $L(x,y^\frown\vec g)$ and $L(y,x^\frown\vec g)$.
    Since $x\E y$, in both cases, the learner has to converge to either $0$ or,
    if $\vec g$ contains a $g_i$ such that $g_i\E x$, some other index.
    Hence, we have two possibilities to consider:
    
    \emph{Case 1:}
    $L(x,y^\frown\vec g)$ or $L(y,x^\frown\vec g)$ converge to $0$. Suppose
    without loss of generality that the latter is the case. Fix $n_0$ such that
    for all $m>n_0$ the value of $L(y,x^\frown\vec g,m)$ doesn't change. Then,
    by the mutual genericity of $\vec g$ relative to $x\oplus y\oplus a$ there exists a finite $\vec p\in {2^{<\omega}}^{<\omega}$ such that
    \[ p\forces (\forall m>n_0) (L(y,x^\frown \dot{\vec g},m)\downarrow\implies L(y,x^\frown \dot{\vec g},m)=0) \]
    By the definability of forcing we get that $x$ and $y$ satisfy
    \begin{equation} 
    \exists \vec p, n_0 (\forall \vec q \leq \vec p)(\forall m>n_0) (L(y,x^\frown\vec q,m)\downarrow\implies L(y,x^\frown\vec q,m)=0).
    \tag{$\ast$}\label{s2:eq1}
    \end{equation}
    By the same reasoning, if $L(x,y^\frown\vec g)$ converges to $0$, then the following is satisfied by $x$ and $y$
    \begin{equation}
    \exists \vec p, n_0 (\forall \vec h \leq \vec p)(\forall m>n_0) (L(x,y^\frown\vec q,m)\downarrow\implies L(x,y^\frown\vec q,m)=0).
    \tag{$\dag$}\label{s2:eq2}
    \end{equation}

    \emph{Case 2:}
    $L(x,y^\frown\vec g)=j_1$ and $L(y,x^\frown\vec g)=j_2$ for some indices $j_1,j_2\neq 0$. Again, by similar genericity arguments as in Case 1, this is forced by some finite $\vec p_1$, respectively, $\vec p_2$. Note that then for any $h_1\succ \vec p_1(j_1)$, $h_2\succ \vec p_2(j_2)$,
    \[ L(x,y^\frown g_0,\dots g_{j_1-1}h_1 g_{j_1+1}\dots)=j_1 \text{ and }L(y,x^\frown g_0,\dots g_{j_2-1}h_2 g_{j_2+1}\dots)=j_2.\]
    Now, fix $h\succ \vec p_1(j_1)$ and $\vec h\succ \vec p_2(j_2)$ (i.e., every element of $\vec h$ extends $p_2(j_2)$) mutually
    sufficiently generic relative to $a$ and look at $L(h, \vec h)$. By transitivity of
    $\E$ it will stabilize to some $k$, and as $h,\vec h$ are mutually generic,
    this is forced. Hence, $x$ and $y$ satisfy
    \begin{equation}
      \begin{split}
        \exists n_0 \exists \vec p_1, \vec p_2 \exists j_1,j_2 (\forall n>n_0)
        (\forall \vec q\leq \vec p_1)  \left(L(x,y^\frown \vec
        q,n)\downarrow\right. &\left.\implies L(x,y^\frown \vec q,n)=j_1\right)\\
        \land (\forall \vec q\leq \vec p_2) \left(L(y,x^\frown \vec
        q,n)\downarrow\right. &\left.\implies L(y,x^\frown \vec q,n)=j_2\right)\\
        \land \exists k (\exists r\leq\vec p_1(j_1)) (\exists\vec r \leq
        \vec p_2(j_2)) (\forall n>n_0)\qquad\quad&\\
        (\forall q\leq r)(\forall \vec q\leq \vec r)
        \left( L(q,\vec q,n)\downarrow\right. &\left.\implies L(q,\vec q,n)=k\right)
      \end{split}
      \tag{$\ddag$}\label{s2:eq3}
    \end{equation}
      We claim that the disjunction of \cref{s2:eq1}, \cref{s2:eq2}, and
  \cref{s2:eq3} defines $\E$. As argued above, if $x\E y$, then one of the
  disjuncts is satisfied. On the other hand, if $x\not \E y$, then
  \cref{s2:eq1,s2:eq2} can not be satisfied. However, it might be that the
  learner gives a false positive in the third conjunct of \cref{s2:eq3} and thus
  $x,y$ satisfy \cref{s2:eq3}. The following claim disposes of this possibility.
    \begin{claim}
        If $g,\vec g$ is a sequence of mutually sufficiently generics relative to $a$ and $L(g,\vec g)=k$ is forced by $p,\vec p$, then $L(g,\vec g)$ does not produce false positives for any $h,\vec h\succ p,\vec p$.
    \end{claim}
    \begin{proof}
        Towards a contradiction assume that $L(h,\vec h)=k$ is a false positive where without loss of generality $k<|\vec p|$. By definition, this is only possible if $h\not \E h_i$ for all $i\in\omega$.  As $L(h,\vec h)=k$ is forced by $p,\vec p$ we have that $L(h,h_0,\dots h_{|\vec p|},h,h_{|\vec p|+1},\dots)=k$. But, by reflexivity, $h\E h$ and thus $h\E h_k$, contradicting our assumption. 
    \end{proof}
    
    ($\Leftarrow$) Suppose that there is an $X\in\Cantor$ and an $X$-computable
    formula $\phi$ such that for all $x,y$ we have $x\E y$ if and only if
    $\exists n \forall m\ \phi(x,y,n,m)$. The construction of an $X$-computable
    learner for $\E$ given input $x,\vec y$ is now straightforward. The
    learner works with an enumeration $(a_1,b_1),(a_2,b_2)\ldots$ of all pairs
    of natural numbers. Given $(a_i,b_i)$, the learner outputs $a_i$ until it
    finds a witness for a failure of the universal formula, i.e., until it finds
    $m$ such that $\neg\phi(x,\vec y(a_i),b_i,m)$. If this happens, the learner
    proceeds to the next pair. Now, if $x\E y$, then for some $a_i$ and $b_i$
    we have $\forall m\ \phi(x,\vec y(a_i),b_i,m)$ and the learner converges on $a_i$.
\end{proof}
Recall that an equivalence relation $E\subseteq X^2$ is \emph{reducible} to an
equivalence relation $F\subseteq Y^2$, if there is a function $f: X\to Y$ such
that for all $x_1,x_2\in X$, $x_1\E x_2$ if and only if $f(x_1)\mathrel{F} f(x_2)$. The
relation $\E$ is \emph{Borel (continuously) reducible} to $\mathrel{F}$ if $f$ is Borel
(continuous).

\cref{thm:sigma2ifflearnable} is in stark contrast to the result by Bazhenov,
Cipriani, and San Mauro~\cite{bazhenov2023learning} that a countable class of structures
is explanatory learnable if and only if it is continuously reducible to $E_0$,
the eventual equality relation on $2^\omega$. While $E_0$ is $\bSigma^0_2$ by its standard
definition
\[ x\E_0 y \iff \exists n(\forall m>n) x(m)=y(m),\]
there are many $\bSigma^0_2$ equivalence relations
that are not Borel reducible to $\E_0$. An important example is the equivalence
relation $E_\infty$, the shift action of $F_2$---the free group on
$2$ generators---on
$2^{F_2}$, which is a Borel-complete countable equivalence relation and thus
vastly more complicated than $\E_0$ from the point of view of Borel
reducibility~\cite{dougherty1994}.

Notice that the definition of the learner from a $\bSigma^0_2$ description of
$\E$ in the proof of \cref{thm:sigma2ifflearnable} produces a learner that does
not give false positives, and that we can get $\bSigma^0_2$ definitions of
learnable equivalence relations regardless if their learners produce false
positives or not. Thus, we obtain the following.
\begin{corollary}\label{cor:falsepositives}
 An equivalence relation $E$ is learnable if and only if it is learnable without false positives.
\end{corollary}
Let $\Gamma$ be a class of equivalence relations and say that an equivalence
relation $E\in \Gamma$ is Borel (continuous) $\Gamma$-complete if
every $F\in \Gamma$ Borel (continuously) reduces to $E$.
One of the most tantalizing questions about a class of equivalence relations
$\Gamma$ is whether there is a Borel (continuous) \emph{$\Gamma$-complete} object. We thus ask the following.
\begin{question}
  Does there exist a Borel (continuous) learnable complete equivalence relation?
\end{question}
By \cref{thm:sigma2ifflearnable} one could equivalently ask whether there is
a $\bSigma^0_2$-complete equivalence relation. It turns out that this is a
difficult and long-standing open problem in descriptive set
theory~\cite{lecomte2020complexity}. 

On the other
hand, given a Polish space $X$,
recall that $A\subseteq X$ is $K_\sigma$ if $A$ is
a union of compact subsets of $X$. Rosendal~\cite{rosendal2005} showed that
there exists a Borel complete $K_\sigma$ equivalence relation. One 
example of a complete $K_\sigma$ equivalence relation is the oscillation
relation $O$~\cite{rosendal2005} which can be defined on $2^\omega$ by
\[ \begin{split}
  x\mathrel{O} y \iff&\\ \exists N \forall n,m &\left((\forall
j\in (n,n+m)) x(j)=0 \implies |\{ j\in (n,n+m): y(j)=1\}|<N\right)\\
 \land &\left((\forall j\in (n,n+m)) y(j)=0 \implies |\{ j\in (n,n+m):
x(j)=1\}|<N \right)\end{split}.\]
\begin{remark}
  The oscillation equivalence relation is usually defined on the space of
  infinite increasing sequences of natural numbers $[\omega]^\omega$. Our definition of $O$ above
  is not equivalent, as we have one additional equivalence class consisting of
  all finite strings. However, the function mapping $(a_0,\dots)$ to the
  characteristic function of the associated set $\{ a_i: i\in\omega\}$ is a continuous embedding of the oscillation
  equivalence relation on $[\omega]^\omega$ into $O$.
\end{remark}
 Since every $\bSigma^0_2$ relation on a compact space is $K_\sigma$ we get the
 following. \begin{proposition}
   The oscillation equivalence relation $O$ is Borel complete among learnable
   equivalence relations on $2^\omega$.
 \end{proposition}
\subsection{Uniform BC-learnability}
\begin{definition}\label{def:bclearning}
  Let $E$ be an equivalence relation on a Polish space $X$ and assume $\omega$ is equipped with the discrete topology. 
  We say that $E$ is \emph{(uniformly) BC-learnable}, or just \emph{BC}-learnable, if there are continuous
  functions $l_n: X^\omega\times X\to \omega$ such that for $x\in X$ and
  $\vec x=(x_i)_{i\in\omega}\in X^\omega$, if $xEx_i$ for some $i\in\omega$,
  then for almost all $n$, $xEx_{l_n(\vec x, x)}$. We call the function
  $L(x,\vec x, n)=l_n(\vec x, x)$ a \emph{BC-learner} for $\E$.
\end{definition}
The rest of this section is devoted to showing that the BC learnable and learnable
equivalence relations coincide. This will follow from a series of lemmas.
\begin{lemma}\label{lem:bcs2}
  If $\E$ is BC-learnable, then every equivalence class is $\bSigma^0_2$.
\end{lemma}
\begin{proof}
  Suppose that $\E$ is BC-learnable and that $[x]_E$ is not open. Pick $y\in
  [x]_E\setminus int([x]_E)$, where $int([x]_E)$ denotes the interior of $[x]_E$. Then for every $n$, there is $b_n\in [[y\restrict
  n]]\setminus [x]_E$. Now, given any $z$ run the following procedure.
  
  We will describe a procedure for each $s\in \omega$, where at each stage
  $s$ we will run $L(y;z,a^{s-1}_1,a^{s-1}_2,\dots;s)$. Here, the $a^s_j$ are either
  prefixes of $y$ or they are equal to $b_n$ for some $n$. To start, let
  $a^{-1}_j=\emptyset$ for all $j$.

  Assume we have defined $i_{s-1}$ and $a^{s-1}_j$ for all $j$ and that $k$ is
  largest such that for all $j<k$, $a^{s-1}_j$ is equal to some
  $b_n$. For all $j<k$, set $a^{s}_j=a^{s-1}_j$. Then let
  $(c_{k},\dots)$ be the lexicographical minimal sequence such that
  $a^{s-1}_j\preceq c_j\prec y$ for all $j\geq k$
  and for some $i_s$,
\[ L(y;z,a^{s}_0,\dots,
a^{s}_{k-1},c_{k},\dots;s)\downarrow=i_s.\]
Now, if $i_s=i_{s-1}=0$ or $s=0$, then let $a^{s}_j=c_j$ for all $j>k$. Otherwise,
for all $j>k$, if there is $i>j$ such that $a^{s-1}_j\neq \emptyset$ or $j\leq i_s$, let $a^s_j$
be the unique $b_n\succ c_j$. If not, let $a^{s}_j=a^{s-1}_j$. Notice that there
can only be finitely many $a^{s-1}_j\neq \emptyset$ as $L(-;-;s)$ is continuous. This finishes the
description.

Say that $a_i=\lim_{s} a^s_i$ for all $i$, and consider $L(y;z,a_0,\dots)$. We
claim that $L(y;z,a^0,\dots)=0$ if and only if $z\in [y]_{E}=[x]_E$. Suppose
$z\in [y]_{\E}$, then there is a least stage $s_0$ such that for all $s>s_0$
$L(y;z,a_0,\dots;s)=i$ implies $a_i\E y_{\E}$. Suppose that $i\neq 0$. Then at
stage $s_0+1$, we declare $a_i=b_n\not\in
[y]_{\E}$ contradicting the learnability of $\E$. On the other hand, if
$L(y;z,a_0,\dots)=0$, then there is a least $s_0$ such that for all $s>s_0$,
$L(y;z,a_0,\dots;s)=0$. Hence, for some large enough $t$, $a_t=y$ and thus $z\in [y]_E$.

To see that $[x]_E$ is $\bSigma^0_2$, let $R$ be the predicate recursive in
$y\oplus \bigoplus b_n$ such that $R(z,s)=i$ if and only if
$L(y;z,a^0_s,\dots;s)=i$. Then $z\in [x]_E$ if and only if $\exists m(\forall
n>m) R(z,n)=0$.
  
\end{proof}
The definition of the $\E$-classes in \cref{lem:bcs2} is inherently non-uniform.
We proceed by showing that for all but countably many classes we can obtain a
uniform definition, thus obtaining that $\E$ is Borel.
\begin{lemma}\label{lem:bcBorel}
  Let $\E$ be BC-learnable, then $\E$ is Borel.
\end{lemma}
\begin{proof}
  Suppose that $\E$ is an equivalence relation on $2^\omega$. Then there are
  only countably many classes that have non-empty interior and only countably
  many classes that contain elements with only finitely many non-zero bits.
  We call these classes \emph{exceptional} as those are the classes for which we
  will need non-uniform definitions.
  
  Let $[x]_E$ be a non-exceptional class, then the procedure given in
  \cref{lem:bcs2} can be turned uniform by choosing $y=x$ and $b_n=x\restrict
  n$. Let $R'$ be the recursive predicate for the modified procedure in
  \cref{lem:bcs2}, then for $x,y$ in a non-exceptional class, $x\E y$ if and only if
  the modified procedure from \cref{lem:bcs2} returns $0$ in the limit on $x$ and $y$, i.e.,
  \[x\E y \iff \exists m(\forall n>m) R'(x,y,n)=0.\] 
  To see that $\E$ is Borel note that $x\E y$ if and only if $x$ and $y$ are in
  the same exceptional class or neither are exceptional and $x\E y$ via the
  procedure $R'$. This gives a Borel definition as required.
\end{proof}
While the definition given in \cref{lem:bcBorel} is Borel it is not
$\bSigma^0_2$, as saying that neither class is exceptional is $\bPi^0_2$. 
However, combining \cref{lem:bcBorel} with recent results of Lecomte~\cite{lecomte2020complexity} we 
are now ready to show that the BC-learnable and learnable equivalence relations
are the same.

\begin{theorem}\label{thm:bclearning}
A Borel equivalence relation $E$ is \emph{BC}-learnable if and only if $E$ is learnable if and only if $E$ is $\bSigma^0_2$.
\end{theorem}
Lecomte \cite{lecomte2020complexity} exhibited a finite set of non
$\bSigma^0_2$ equivalence relations that are minimal and form an antichain
under continuous reducibility. Furthermore, every Borel equivalence relation
that is not $\bSigma^0_2$ continuously embeds one of these relations. Our proof
of \cref{thm:bclearning} will show that none of Lecomte's relations are
$BC$-learnable. Before we proceed with the proof let us recall Lecomte's result
in more detail. 

The five equivalence relations can be defined as follows. Let $\INF$ refer to
the class of characteristic functions of all infinite subsets of $\omega$.
That is, $\INF=\{x\in 2^{\omega}:\forall n (\exists m>n) x(m)=1\}$. All the relations below are defined on $2^\omega$.
\begin{itemize}

  \tightlist
  \item 
    $x \mathbin{\sim_0} y\iff (x,y\in \INF) \vee (x=y)$
    \item 
    $ x \mathbin{\sim_1} y\iff (x,y\in \INF) \vee (x,y\notin \INF)$
    \item 
    $ x \mathbin{\sim_3} y\iff (x=y) \vee (x,y\in \INF \land (\forall n>0) x(n)=y(n))$
    \item 
    $ x \mathbin{\sim_4} y\iff x \mathbin{\sim_3} y \vee (x(0)=y(0)=1\wedge
    x,y\notin \INF)$
    \item 
    $ x \mathbin{\sim_5} y\iff x \mathbin{\sim_3} y \vee (x(0)=y(0)\wedge
    x,y\notin \INF)$
\end{itemize}
\begin{theorem}[Lecomte \cite{lecomte2020complexity}]
Let $\E$ be a Borel equivalence relation. Then exactly one of the following holds:
\begin{itemize}
    \item $\E$ is $\bSigma^0_2$.
    \item there is $i\in \{0,1,3,4,5\}$ and an injective continuous reduction from
      $\sim_i$ to $\E$.
\end{itemize}
\end{theorem}

\begin{proof}[Proof of Theorem \ref{thm:bclearning}.] 
  Since every learnable
  equivalence relation is also BC-learnable, we only have to deal with the other
  implication. Our proof proceeds by showing that each of Lecomte's equivalence
  relations is not BC-learnable. Since BC-learnability is preserved by
  continuous reductions, this implies that the BC-learnable Borel equivalence
  relations are precisely the learnable Borel equivalence relations.
  Using \cref{lem:bcBorel} we can then infer \cref{thm:bclearning} in its full
  generality.

  It remains to show that none of the $\sim_i$, $i\in\{0,1,3,4,5\}$ are
  BC-learnable. In all five cases, we achieve that by exploiting the fact that
  $\INF$ has no $\bSigma^0_2$ definition. 
  The argument are analogous. We assume that $\sim_i$ is
  BC-learnable and use this assumption to obtain a $\bSigma^0_2$ definition of
  $\INF$, which gives a contradiction. 

  Suppose that $\sim_0$ is BC-learnable by $L$. We claim that the following
  is a $\bSigma^0_2$ definition of $\INF$: 
  \[x\in \INF\iff\exists m (\forall
  n>m) L(x;1^{\infty},f_1,f_2,\ldots,n)=0.\]
  where $f_1,f_2,\ldots$ is an
  enumeration of all elements of $2^{\omega}$ with finitely many ones. Note that $x$
  is either finite and identical to some $f_i$, or infinite and $\sim_0$-equivalent to
  $1^{\infty}$. Hence, if $L$ is a BC-learner for $\sim_0$, then $L$ converges
  on $0$ if $x$ is infinite and to some $k\neq 0$ otherwise. A similar argument
  shows that the above formula gives a $\bSigma^0_2$ definition for $\sim_1$ if we assume that $L$ is a $BC$-learner for
  $\sim_1$. 

  Similarly, if we assume that $L$ is a BC-learner for $\sim_i$, $3\leq i\leq
  5$, then the following formula gives a $\bSigma^0_2$ definition of $\INF$:
  \[x\in \INF\iff
  \exists m (\forall n>m)L(0\concat x;1\concat x,0\concat f_1,0\concat f_2,\ldots,n)=0.\]

\end{proof}

\subsection{Borel learnability}
In modern
descriptive set theory Borel functions arguably play a more prominent 
role than continuous functions. The reason for this is the following
classical change of topology theorem. If $f:(X,\sigma)\to Y$ is Borel, then there is a Polish topology
$\tau\supset\sigma$ such that the Borel sets generated by $\tau$ and $\sigma$
are the same and $f:(X,\tau)\to Y$ is continuous; see~\cite[Chapter
13]{kechris2012} for a proof.

The following modification of uniform learnability might thus look more
natural to a descriptive set theorist.
\begin{definition}\label{def:borellearnability}
  Let $E$ be an equivalence relation on a Polish space $X$ and assume $\omega$ is equipped with the discrete topology. 
  We say that $E$ is \emph{Borel learnable}, if there are Borel 
  functions $l_n: X^\omega\times X\to \omega$ such that for $x\in X$ and
  $\vec x=(x_i)_{i\in\omega}\in X^\omega$, if $x\E x_i$ for some $i\in\omega$,
  then $L(\vec x,x)=\lim l_n(\vec x, x)$ exists and $x \E x_{L(\vec x,x)}$. We refer to the partial function $L$ as a \emph{Borel learner}
  for $E$.
\end{definition}
Combining the above change of topology theorem with the boldface version of
\cref{thm:sigma2ifflearnable} we immediately obtain that an equivalence
relation $E$ on a Polish space $(X,\sigma)$ is Borel learnable if and only if
there is $\tau\supseteq \sigma$ such that $E$ is $\bSigma^0_2$ in $(X,\tau)^2$.
Following Louveau~\cite{louveau1995}, we say that such an equivalence relation
is \emph{potentially $\bSigma^0_2$} and thus get the following.
\begin{theorem}\label{thm:borellearnableiffpotsigma2}
  An equivalence relation $E$ is Borel learnable if and only if $E$ is
  potentially $\bSigma^0_2$.
\end{theorem}

One can also consider the notion of non-uniform Borel learnability by
adapting~\cref{def:nonuniformlearning} in the obvious way. We then get the
following from another classical change of topology theorem.
\begin{proposition}\label{prop:eqclassesBorelsep}
  Let $\E$ be an equivalence relation on a Polish space $X$. Then the following are equivalent.
  \begin{enumerate}
    \item $\E$ is non-uniformly Borel learnable. 
    \item\label{it:Borelsep} For every $\vec x\in X^\omega$, there are sets $S_i$, $S_i\in Borel(X)$ such
  that for every $i\in\omega$, $[x_i]\subseteq S_i$ and $S_i\cap S_j=\emptyset$
  if $x_i\not \E x_j$.
  \end{enumerate}
  Furthermore, $\E$ is non-uniformly learnable by learners not giving false
  positives if and only if $[x]$ is Borel for every
  $x\in X$.
\end{proposition}
\begin{proof}
  Suppose we have a learner $L$ learning $\vec x\in X^\omega$. Using the
  change of topology theorem mentioned above, we get a topology $\tau$ on $X$
  such that the learner $L$ is continuous. Hence, by \cref{prop:eqclassesS2sep},
  there are $S_i$ separating the equivalence classes $[x_i]$ such that every $S_i$ is
  $\bSigma^0_2$ in $\tau$. But then as the Borel sets in $\tau$ and the original
  topology are the same, the $S_i$ are Borel in the original topology on $X$.

  On the other hand, recall another classical change of topology theorem: If
  $(A_i)_{i\in\omega}$ is a countable sequence of Borel sets on $X$, then there
  is a refinement $\tau$ of $X$ generating the same Borel sets, such that every
  $A_i$ is clopen, see again~\cite[Chapter 13]{kechris2012}. Hence, given $S_i$
  as in \cref{it:Borelsep} we get a topology $\tau$ with all $S_i$ clopen and a learner $L$ for $\vec x$ that is
  $\tau$-continuous, and hence Borel.
\end{proof}

\section{Complexity of learnability}\label{sec:complexityoflearning}
For the proofs of the following results, we rely on Louveau's separation theorem and the Spector-Gandy theorem. Louveau's separation theorem says that if $X$ is a recursive Polish space and $A_0,A_1$ are two disjoint $\Sigma^1_1$ sets that are $\bSigma^0_\alpha$ separated for a recursive ordinal $\alpha$, then there is a $\Sigma^0_\alpha(HYP)$ set separating $A_0$ from $A_1$~\cite{louveau1980}.
In particular, if $X$ is $\Delta^1_1$ and $\bSigma^0_\alpha$, then $X$ is $\Sigma^0_\alpha(HYP)$. To see this, just take $A_0=X$ and $A_1=\bar X$.
We will combine this with the classic Spector-Gandy theorem. The version most suitable for our purpose says that a set $X\subseteq 2^\omega$ is $\Pi^1_1$ if and only if there is a recursive predicate $R$ such that
\[ x\in X \iff (\exists z \leq_h x)R(x,z).\]
We suggest~\cite{chong2015} for a detailed discussion of this theorem.
The following result is a consequence of Louveau's separation theorem and \cref{thm:sigma2ifflearnable}.
\begin{lemma}\label{lem:hyplearner}
Let $E$ be a $\Delta^1_1$ equivalence relation and learnable. Then there is a hyperarithmetical learner learning $E$.
\end{lemma}
\begin{proof}
    Recall from \cref{thm:sigma2ifflearnable} that an equivalence relation is learnable by an $x$-computable learner if and only if it is $\Sigma^0_2(x)$. Now by Louveau's separation theorem $\Delta^1_1 \cap \bSigma^0_2=\Sigma^0_2(HYP)$. Thus $E$ being learnable and $\Delta^1_1$ implies that $E$ is $\Sigma^0_2(HYP)$ and so it is learnable by a hyperarithmetical learner.
\end{proof}
We can represent Borel equivalence relations on a Polish space $X$ using their
Borel codes. Recall that a \emph{Borel code} is a labeled tree $T\in
\omega^{<\omega}$ where the leaf nodes are labeled by (the codes) for elements
of 
a fixed subbase of $X$ and internal nodes are labeled by $\cup$ or $\cap$,
indicating that we take unions or intersections of the sets described by the
subtree rooted at that node. We can view labeled trees as elements of
$2^\omega$ and thus can talk about the complexity of the set of Borel
equivalence relations with some property $P$.

\begin{theorem} \label{thm:learnablepi11}
    The set of learnable Borel equivalence relations on Cantor space is $\Pi^1_1$ in the codes.
\end{theorem}
\begin{proof}
  Note that $T$ is a code for a learnable equivalence relation if and only if
  (1) $T$ is well-founded, (2) its labeling is correct, and (3) there is a
  learner learning $E_T$, the equivalence relation represented by $T$.
  The first statement is $\Pi^1_1$, the second arithmetical, and by the
  relativization of \cref{lem:hyplearner}, we have that $E_T$ is learnable if
  and only if there exists a learner hyperarithmetic in $T$. By the
  Spector-Gandy theorem this is a $\Pi^1_1$ statement.
\end{proof}
We will establish that the set of learnable Borel equivalence relations on
Cantor space is $\bPi^1_1$-complete in the codes. Let us quickly recall the
definitions.
Let $X$ and $Y$ be Polish spaces, and $A\subseteq X$, $B\subseteq Y$. Then
$A$ is \emph{Wadge reducible} to $B$, $A\leq_W B$ if there exists a continuous
function $f:X\to Y$ such that for all $x\in X$, $x\in A$ if and only $f(x)\in
B$. For a pointclass class $\Gamma$, and $A\in \Gamma$, $A$ is
\emph{$\Gamma$-complete} if for every $B\in\Gamma$, $B\leq_W A$.
We are now ready to pin down the complexity of the learnable Borel equivalence
relations.

\begin{theorem}\label{thm:pi02learnablecomplete}
   The set of learnable $\bPi^0_2$ equivalence relations on Cantor space is
   $\bPi^1_1$-complete in the codes.
\end{theorem}
\begin{proof}
Recall that the set of well-founded trees in $\omega^{<\omega}$ is
$\bPi^1_1$-complete. We reduce it
to the set of codes of learnable $\bPi^0_2$ equivalence relations as follows.
Given $x\in 2^\omega$ we might view it as a disjoint union $x=x_1\oplus
x_2$ where $x_1(n)=x(2n)$ and $x_2(n)=x(2n+1)$. We furthermore write $p_x$ for
the principal function of $x$, that is the function enumerating the set of
natural numbers defined by $x$ in order. Notice that $p_x\in \omega^\omega$ if and only if $x$ is infinite, and thus this is already a $\Pi^0_2$
statement. Also recall that the set $\INF=\{\chi_X: X\subseteq \omega,
|X|=\infty\}$ is $\bPi^0_2$-complete. We are now ready to define our reduction. 
Given a tree $T\subseteq \omega^{<\omega}$ define a $\Pi^0_2$ equivalence relation $E_T$
by 
\[ x\mathbin{E_T} y\iff x=y \lor\left(p_{x_1},p_{y_1}\in [T] \land
x_2,y_2\in\INF\right).\]
Notice that if $T$ is well-founded, then $E_T=id_{2^\omega}$. On the other
hand, if $T$ is ill-founded, then we claim that $E_T$ is not learnable. Fix
$x\in [T]$, and consider $[\langle x,1^\infty\rangle]_{E_T}$, the equivalence class of $\langle x, 1^\infty\rangle$. If $E_T$ were learnable, then by \cref{thm:borellearnableiffpotsigma2}, $[\langle x,1^\infty\rangle]_{E_T}$ would be $\Sigma^0_2(x)$. But then as 
\[ y\in \INF \iff \langle x,y\rangle\in [\langle x,1^\infty\rangle]_{E_T}\]
we would get that $\INF$ is $\Sigma^0_2(x)$, a contradiction.
Thus the function $T\mapsto E_T$ is a Borel reduction
(even continuous reduction) from
the set of well-founded trees to the set of codes of $\bPi^0_2$ equivalence
relations that are learnable.
\end{proof}
Since the equivalence relation $E_T$ produced in the above proof is
$id_{2^\omega}$ if $T$ is well-founded, and this equivalence relation is
learnable by a computable learner, we obtain the following effective analog to
\cref{thm:pi02learnablecomplete}.
\begin{corollary}\label{cor:complearnercomplete}
  The set of $\bPi^0_2$ equivalence relations that are learnable by a computable
  learner is $\bPi^1_1$-complete in the codes.
\end{corollary}
\cref{thm:learnablepi11} directly implies the following result.
\begin{corollary}\label{cor:learnablepi11complete}
The set of learnable equivalence relations is $\bPi^1_1$ complete in the
codes.
\end{corollary}
Using the recursion theorem to identify computable codes for Borel equivalence
relations with their indices we get similar results for the index sets of codes
for learnable equivalence relations.
\begin{corollary}\leavevmode
  \begin{enumerate}
\item The set of indices of codes for learnable $\Delta^1_1$ equivalence
  relations is $\Pi^1_1$-complete.
    \item
  The set of indices of codes for computably learnable $\Delta^1_1$ equivalence relations learnable by a
  computable learner is $\Pi^1_1$-complete.
\end{enumerate}
\end{corollary}
\section{Case study}
In this section, we study whether various equivalence relations arising in logic
are learnable, and, since learnability is closely tied to the Borel hierarchy,
place these equivalence relations in this hierarchy.

\subsection{Equivalence relations arising in computability
theory}
\label{sec:casestudy}
\begin{theorem}\label{thm:turingnotlearnable}
    Turing equivalence is not learnable. In fact, it is not
    $\bPi^0_3$.\footnote{We thank Uri Andrews for pointing out that our proof that
      Turing equivalence is not learnable can be modified to show that it is not
    $\bPi^0_3$.}
\end{theorem}
\begin{proof}
Assume towards a contradiction that $\equiv_T$ is definable by the $\Pi^0_3$ formula
\[\psi(x,y)= \forall k\exists m\forall n \phi(x,y,k,m,n).\]
Consider a sufficiently mutual generic pair $g_1,g_2$. Then, as $g_1,g_2$ are not
Turing equivalent, there is $(p_1,p_2)$ such that $(p_1,p_2)\forces \exists k
\forall m\exists n \neg\phi(\dot g_1,\dot g_2,k,m,n)$. In particular, there is $k_1$
such that the set 
\[ Q=\{(q_1,q_2): (q_1,q_2) \forces \forall m\exists n \neg \phi(\dot g_1,\dot g_2,
k_1,m,n)\}\]
is dense below $(p_1,p_2)$. Using this fact we can build two recursive reals
$(h_1,h_2)\succ (p_1,p_2)$ such that $\psi(h_1,h_2)$. 

We build $h_1$ and $h_2$
in stages. Say at stage $s$ we have built the finite strings $h_1^s$ and
$h_2^s$. Let $h_1^{s+1}$ and $h_2^{s+1}$ be the lexicographical
minimal pair 
$(q_1,q_2)$ that extends $(h_1^s,h_2^s)$ and satisfies $\exists n \neg\phi(\dot
g_1,\dot g_2,s,n)$. We can always find
such a pair as $Q$ is dense below $(p_1,p_2)$ and the search is recursive since
by the definition of the forcing relation the existential quantifier must be
witnessed by $n< \min \{|q_1|,|q_2|\}$.

As $h_1,h_2$ are recursive and $\neg \psi(h_1,h_2)$ we have that $\psi$ does not
define $\equiv_T$, a contradiction. So $\equiv_T$ is not $\Pi^0_3$.
Relativizing this argument shows that it is not $\bPi^0_3$ and thus also not
learnable.
\end{proof}
By Wadge's lemma~\cite{wadge1983} any set that is not $\bPi^0_3$ is
$\bSigma^0_3$-hard. Hence, we obtain the following corollary.
\begin{corollary}
  Turing equivalence is $\bSigma^0_3$-complete.
\end{corollary}
Turing equivalence was a good candidate for being learnable, as it is naturally
$\Sigma^0_3$. Coarser equivalence relations such as arithmetic equivalence or
hyperarithmetic equivalence are quite far from being learnable. The former is
not arithmetical and the latter is not even Borel. We assume that these results
are well-known but let us quickly summarize the overall ideas needed to show
this, which are quite
similar to the ones used in \cref{thm:turingnotlearnable}.

For arithmetic equivalence, assume it was $\Sigma^0_n(x)$ for some $n\in \omega$ and a real $x$. Then, in the product forcing, there is $(p_1,p_2)$ such that every
$n$-generic pair extending $(p_1,p_2)$ is not arithmetic equivalent. However,
one can show that there are arithmetically equivalent $n$-generics, contradicting the above. The proof
for hyperarithmetic equivalence is similar.

Let us now consider the equivalence relations of $m$ and $1$-equivalence. 
Recall that for two reals $x,y$, $x\equiv_m y$ if and only if there exists a
total computable function $\phi_e$ such that $x(n)=1$ if and only if
$y(\phi_e(n))=1$. The two reals $x$, $y$ are $1$-equivalent, $x\equiv_1 y$ if
$x\equiv_m y$ via a 1-1 function. Both of these relations are naturally
$\Sigma^0_3$, $m$-equivalence being defined by
\begin{equation}\tag{*}\label{eq:s31redu} x\equiv_m y \iff \exists e \ \phi_e
\text{ total} \land \forall n
  (x(n)\leftrightarrow y(\phi_e(n)))
\end{equation}
and $1$-equivalence being defined analogously.

\begin{theorem}\label{thm:mequivnotcomplearn}
   Both $\equiv_m$ and $\equiv_1$ are not $\Pi^0_3$.
\end{theorem}
\begin{proof}
  The proof of this theorem is almost the same as the proof of
  \cref{thm:turingnotlearnable}.
  Suppose towards a contradiction that $\equiv_1$ is given by the $\Pi^0_3$ formula
  \[ \phi(x,y) \iff \exists l\forall m\exists n R(x,y,l,m,n).\]
  Then $\phi(\dot g_1,\dot g_2)$  cannot be forced by any $(p_1,p_2)$ because
  any two mutually generic $g_1\succ p_1$, $g_2\succ p_2$ are Turing
  incomparable and thus also $g_1\not\equiv_1 g_2$. So again, we can find
  $(p_1,p_2)$ such that $(p_1,p_2)\forces\neg \phi(\dot g_1,\dot g_2)$. Interleaving
  the construction of recursive $h_1\succ p_1$ and $h_2\succ p_2$ with
  construction steps that ensure that $h_1$ and $h_2$ are both infinite,
  coinfinite we get that $h_1\equiv_1 h_2$ as they are characteristic functions
  of recursive infinite, coinfinite sets but $\neg \phi(h_1,h_2)$. Hence, $\equiv_1$ can not be defined
  by $\phi$, a contradiction.
\end{proof}
The reals built in the proof of \cref{thm:mequivnotcomplearn} are 1-equivalent
in the $\Sigma$-outcome and not even Turing equivalent in the $\Pi$-outcome.
This immediately gives that no equivalence relation
that is finer than Turing equivalence and has $\equiv_1$ as a subequivalence
relation is $\Pi^0_3$.
\begin{corollary}
  Say $\E$ is an equivalence relation such that $\equiv_1\subseteq
  \E\subseteq \equiv_T$, then $\E$ is not $\Pi^0_3$.
\end{corollary}

The fact that we could assume that the relation $R$ in the proof of
\cref{thm:mequivnotcomplearn} was recursive, allowed us to build recursive and
infinite coinfinite reals in the $\Sigma$-outcome, thus ensuring that they are
$1$-equivalent.
Without this assumption, we cannot guarantee anymore that the elements built are
both recursive, and thus our proof fails. In fact, if we allowed
$\mathrm{TOT}=\{e:\phi_e \text{ total}\}$ as an oracle, then the definition of
$m$-equivalence in \cref{eq:s31redu} becomes $\Sigma^0_2$. 

\begin{theorem}\label{thm:mequivtotlearnable}
  Both $\equiv_m$ and $\equiv_1$ are $\Sigma^0_2(\mathrm{TOT})$. Hence, they are
  learnable by a $\mathrm{TOT}$-computable learner.
\end{theorem}

\subsection{Jump hierarchies}
Let us recall the notion of a jump hierarchy. Roughly speaking, the Turing jump
assigns to a set $X$ its corresponding halting set $X'$. 
This operation can be iterated to obtain
"halting sets relative to halting sets" $X^{(n)}$ and so on---first, via
finite iterations and then, via transfinite recursion. A jump hierarchy is a set
that encodes in its columns a countable number of iterations of the Turing
jump. In this way, any computable well-order 
 gives rise to a jump hierarchy. 
 \begin{definition}
   Given an infinite computable linear order $\prec$ on natural numbers let
   $[{\prec}n]=\{ m: m\prec n\}$. A set
$H\subseteq\omega$ is a \emph{jump hierarchy supported on $\prec$} if
\[
  \forall n\in\omega\ H^{[n]}=(H^{[{\prec} n]})'.    
  \]
\end{definition}
We can immediately observe that this definition is $\Pi^0_2$. If $\alpha=(\omega,\prec)$ is isomorphic to a computable ordinal, then the corresponding jump hierarchy is uniquely defined. In such cases, we abuse the notation and use $\emptyset^{(\alpha)}$ to denote the unique jump hierarchy supported on $\prec$. In fact, if $\alpha$ is a computable ordinal, then the corresponding jump hierarchy $\emptyset^{(\alpha)}$ is a $\Sigma^0_{1+\alpha}$-complete set. 

Jump hierarchies allow us to exemplify relatively simple, i.e., defined by a
$\Pi^0_2$ formula, equivalence relations, which are uniformly learnable but
which require learners with computational power at arbitrarily high levels of the
hyperarithmetic hierarchy. To this end, given a computable linear order $\prec$
and the corresponding $\Pi^0_2$ class $\mathbf{H}$ consisting of jump
hierarchies supported on $\prec$ we define $\equiv_{\prec}$ as follows. 
\[x\equiv_\prec y \iff x=y\mbox{ or }x^{[1]}=y^{[1]}=z\in \mathbf{H} .\]
\begin{proposition}\label{prop:wfjhlearn}
Given a presentation $(\omega,\prec)$ of an ordinal $\alpha$, the equivalence relation $\equiv_{\prec}$ is uniformly learnable and for every $z$, if $\equiv_{\prec}$ is $z$-learnable, then $z\geq_T 0^{(\alpha)}$.
\end{proposition}
\begin{proof}
First, since the jump hierarchy supported on $\prec$ is unique, it is straightforward to see that $\equiv_\prec$ is learnable.
Now, using the definition of $\equiv_{\prec}$ and our characterization
of learnability (Theorem \ref{thm:sigma2ifflearnable}), we notice that
$x=\emptyset^{(\alpha)}$ if and only if
$\emptyset\oplus x\equiv_{\prec}\omega\oplus x$ if and only if $\exists n\forall m
\phi(\emptyset\oplus x,\omega\oplus x,n,m)$ where $\phi$ is $\Delta_1^0(z)$ and
$z$ computes a learner for $\equiv_{\prec}$. Fixing a witness $n$ for $\phi$ we
obtain a $\Pi_1^0(z)$ definition of $\emptyset^{(\alpha)}$ and therefore,
$\emptyset^{(\alpha)}$ is the unique path in a $z$-computable binary tree. Each
isolated path in a $z$-computable binary tree is $z$-computable, hence $z$
computes $\emptyset^{(\alpha)}$.
\end{proof}
While the notion of jump hierarchy was originally defined with well-founded
linear orders in mind---we start with the degree of computable functions and
iterate the jump from there---a surprising consequence of the Gandy basis theorem is that there exist nonstandard jump hierarchies, supported on ill-founded linear orders. However, these are no longer unique. Indeed, Harrison \cite{harrison1968recursive} noted that when an ill-founded linear order supports a jump hierarchy, then it supports a continuum many of them. Moreover, none of these non-standard jump hierarchies can be hyperarithmetic. Indeed, one can show that if a jump hierarchy contains an infinite descending sequence, then it must compute all hyperarithmetic reals. In fact, one can go further, and observe that if $\prec$ is ill-founded and supports a jump hierarchy, then no infinite $\prec$-descending sequence is hyperarithmetic. Thus, we can make the following observation.
\begin{proposition}\label{prop:ifnolearn}
Given an ill-founded computable order $\prec$, which supports a jump hierarchy, the equivalence relation $\equiv_{\prec}$ is not uniformly learnable.
\end{proposition}
\begin{proof}
  As in the proof of \cref{prop:wfjhlearn} we assume that $z$ learns
  $\equiv_\prec$ and apply Theorem \ref{thm:sigma2ifflearnable} to argue that
  there is a $\Pi^0_1(z)$ definition of some non-empty subset of the class of
  jump hierarchies supported on $\prec$. Therefore, there exists a member of
  this class which is arithmetic in $z$---in fact, even $z'$-computable.
  However, by Louveau's separation theorem, $z$ may be assumed to be hyperarithmetic, as $\equiv_\prec$ is hyperarithmetic, in fact even $\Pi_2^0$. This contradicts the fact that no jump hierarchy supported on an ill-founded order is hyperarithmetic.
\end{proof}
One consequence of \cref{prop:ifnolearn} is that for equivalence relations of
the form $\equiv_\prec$ learnability identifies exactly the well-founded part of
$\prec$. More formally, $\equiv_{{\prec} n}$ is learnable if and only if ${\prec} n$
is well-founded. This observation is closely connected to the following result, which is
interesting in its own right.
\begin{theorem}
Let $\prec$ be an ill-founded computable order that supports a jump hierarchy. The well-founded part $\mathcal{O}_\prec$ of $\prec$ is $\Pi^1_1$ but not $\Sigma^1_1$ and not $\Pi^1_1$-complete.
\end{theorem}
\begin{proof}
The well-founded part of any computable linear order $\prec$ is clearly
$\Pi^1_1$ as it is defined by 
\[ n\in \mathcal O_{\prec} \iff (\forall \bar a \in {\prec}n^\omega) \exists i\ 
a_{i}\prec a_{i+1}\]
and ${\prec} n$ is uniformly computable.
Since $\prec$ supports a jump hierarchy, it contains no infinite descending
hyperarithmetic sequence. In particular, $\overline{\mathcal{O}_\prec}$ cannot
be hyperarithmetic, i.e., $\Delta^1_1$, and hence, $\mathcal{O}_\prec$ is not
$\Delta^1_1$, which together with it being $\Pi^1_1$ implies that it is not $\Sigma^1_1$.

To see that $\mathcal{O}_\prec$ is not $\Pi^1_1$-complete suppose towards a
contradiction that the $\Pi^1_1$-complete set $\mathrm{WO}=\{e: \phi_e \text{ computes a diagram of a well-ordered set}\}$ reduces to $\mathcal{O}_\prec$ via a computable function
computed by $g$. For simplicity, we assume that $\prec$ is a strict ordering.
We use $g$ to define a computable linear ordering $\leq_{a,b}$ uniformly for each pair of natural numbers $a$ and $b$. The ordering is defined by induction---at each stage $s$, $\leq_{a,b}$ is a finite linear ordering of
size $s$, starting from an empty ordering at $s=0$. Suppose we have defined the ordering on the first $s$ natural numbers and we are at stage $s+1$ of the construction. If
$g_{s}(a)\prec
g_{s}(b)$, then we add $s+1$ as the smallest element so far in $\leq_{a,b}$. Otherwise, i.e., if $g_{s}(a)\prec g_{s}(b)$ or
one of the two computations did not halt after $s$ steps, then we add $s$ as the largest element so far in $\leq_{a,b}$. This finishes the construction.

Now, we use Smullyan's double recursion theorem to argue that there exists $e_0$
and $e_1$ such that $\phi_{e_0}$ computes the diagram of $\leq_{e_0,e_1}$ and
$\phi_{e_1}$ computes the diagram of $\leq_{e_1,e_0}$. Note that
$\leq_{e_0,e_1}$ is ill-founded if and only if $\leq_{e_1,e_0}$ is well-founded.
Suppose that $g(e_0)\prec g(e_1)$. This means that $\leq_{e_0,e_1}$ is
ill-founded. However, by our assumption $g$ is an appropriate reduction, and
hence, $g$ maps $e_0$ outside of $\mathcal{O}_\prec$. But since $\leq_{e_1,e_0}$
is well-founded, $g(e_1)\in\mathcal{O}_\prec$. This is a contradiction, as the well-founded part of $\prec$ is closed downwards. Similarly, we obtain a contradiction for the case when $g(e_1)\prec g(e_0)$.

\end{proof}

\subsection{Model theory}
As mentioned in the introduction, the work in this paper was inspired by the work
of Fokina, Kötzing, and San Mauro~\cite{fokina2019limit} and others on learning
the isomorphism relation on countably many isomorphism types. One of the first
questions that arises is whether it is possible to learn the isomorphism
relation on the whole space of structures. This should of course be impossible,
as, in general, the isomorphism relation is not even Borel. The following
proposition makes this observation precise in terms of non-uniform learnability.
\begin{proposition}\label{prop:vocabnotlearnable}
  Let $\tau$ be a countable vocabulary containing at least one relation symbol. Then $Mod(\tau)$ is not non-uniformly
  learnable. 
\end{proposition}
\begin{proof}
We will assume without loss of generality that $\tau$ contains exactly one unary
relation symbol $R$.

Consider structures $\A_n$ such that $R^{\A_n}(m)$ if and only if $m\leq n$, and
a structure $\A_\infty$ where $\{m: R^{\A_{\infty}}(m)\}$ is an infinite, coinfinite subset
of $A$. Given $x\in 2^\omega$, let $\B_x$ be such that $R^{\B_x}(2n) \iff
x(n)=1$ and $\neg R^{\B_x}(2n+1)$ for all $n$. Clearly, the map $x\mapsto
\B_x$ is continuous and $\B_x\cong \A_\infty$ if and only if $x\in \INF$.

Now, assume there was a learner $L$ for the sequence $(\A_\infty,\A_1,\dots)$,
i.e., there are $\bSigma^0_2$ sets separating the isomorphism classes. Then, we
would get that $x\in \INF$ if and only if $\B_x\in S$ where $S$ is the
$\bSigma^0_2$ set containing $Iso(\A_\infty)$. But then $\INF$ would be
$\bSigma^0_2$, contradicting that $\INF$ is $\bPi^0_2$-complete.
\end{proof}
\begin{remark}
Various examples of countable classes of structures $C$ that are explanatory
learnable are known~\cite{bazhenov2020learning}. Thus,
\cref{prop:vocabnotlearnable} fails if we restrict our attention to $\tau$-structures satisfying a fixed 
$L_{\omega_1\omega}$ sentence $\phi$.
\end{remark}
The following is an immediate corollary of \cref{prop:vocabnotlearnable} and
\cref{prop:eqclassesS2sep}. It generalizes the result by Miller that no
structure in relational language has a $\Sinf{2}$ Scott
sentence~\cite{miller1983}.
\begin{corollary}
  There is no vocabulary $\tau$ such that for all $\A_0,\dots \in Mod(\tau)$
  there are $\bSigma^0_2$ sets $S_0,\dots$ such that $\A_i\subseteq S_i$ and
  $S_i\cap S_j=\emptyset$ if and only if $\A_i\not\cong \A_j$.  
\end{corollary}

If we consider Borel learnability, then the picture becomes a bit more
interesting. Hjorth and Kechris~\cite{hjorth1996} showed that an equivalence
relation induced by the Borel action of a closed subgroup of $S_{\infty}$, the permutation group of
the natural numbers, is potentially $\bSigma^0_2$ if and only if it is
\emph{essentially countable}, i.e., Borel
reducible to a countable equivalence relation. Since the isomorphism relation on
any class of structures is an orbit equivalence relation induced by
$S_{\infty}$ we get the following as a corollary of
\cref{thm:borellearnableiffpotsigma2}.
\begin{corollary}\label{cor:learnableiffessentctble}
  Given an $L_{\omega_1\omega}$ sentence $\phi$ we have that $\cong_\phi$ is Borel
  learnable if and only if $\cong_\phi$ is essentially countable.
\end{corollary}

\printbibliography
\end{document}